\documentclass[12pt]{amsart}

\usepackage{amsmath, amssymb}
\usepackage{array}
\usepackage[frame,cmtip,arrow,matrix,line,graph,curve]{xy}
\usepackage{graphpap, color, paralist, pstricks}
\usepackage[mathscr]{eucal}
\usepackage[pdftex]{graphicx}
\usepackage[pdftex,colorlinks,backref=page,citecolor=blue]{hyperref}
\usepackage{tikz}
\usepackage{tikz-cd}

\usepackage{enumitem}

\newtheorem{theorem}{Theorem}[section]

\newtheorem{proposition}[theorem]{Proposition}
\newtheorem{corollary}[theorem]{Corollary}
\newtheorem{lemma}[theorem]{Lemma}

\theoremstyle{definition}
\newtheorem{definition}[theorem]{Definition}

\newtheorem{remark}[theorem]{Remark}

\newtheorem*{acknowledgement}{Acknowledgement}

\newtheorem{maintheorem}{Theorem}

\renewcommand{\AA}{\mathbb{A} }

\newcommand{\EE}{\mathbb{E} }
\newcommand{\FF}{\mathbb{F} }
\newcommand{\NN}{\mathbb{N} }
\newcommand{\PP}{\mathbb{P} }
\newcommand{\QQ}{\mathbb{Q} }

\newcommand{\ZZ}{\mathbb{Z} }

\newcommand{\bp}{\mathbf{p} }

\title{The number of solutions of a random system of polynomials over a finite field}
\author{Ritik Jain}
\address{Ritik Jain, Department of Mathematics, University of Connecticut, Storrs, CT 06269}
\email{ritik.jain@uconn.edu}

\begin{document}
\begin{abstract}
We study the probability that a random system of $m$ polynomials in $n$ variables over a finite field $\FF_q$ has exactly $k$ solutions. We show that the distribution over $k$ is binomial with parameters $q^n, 1/q^m$. More generally, in the case that the polynomials have coefficients in a general finite commutative ring $R$, we find that the expected number of solutions is $|R|^{n-m}$. In particular, if $n=m$, this implies that the system of polynomials has exactly one solution on average, regardless of $R$.
\end{abstract}
\maketitle

\section{Introduction}
Fix a collection of uniformly random polynomials $f_1, \cdots, f_m$ over $\FF_q$ in $n$ variables. We study the distribution of the random variable
\[
N(f_1,\dots,f_m)
\;=\;
\#\left\{\mathbf{x}\in \FF_q^n \;\middle|\; f_1(\mathbf{x})=\cdots=f_m(\mathbf{x})=0 \right\},
\]
which encodes the number of distinct solutions of the system. Assuming the sample space for the $f_i$ satisfies the mild largeness condition given in Definition \ref{defn:spaces}, (a), we show (Theorem \ref{thm:main1}) that
\[
N(f_1,\dots,f_m) \sim \mathrm{Bin}\left( q^n, \frac{1}{q^m} \right).
\]
In particular, this implies that
\[
\EE(N(f_1,\dots,f_m)) = q^{n-m}, \qquad \mathrm{Var}(N(f_1,\dots,f_m))= q^{n-m}(1-q^{-m}),
\]
where the expectation is $1$ exactly when $n = m$.
We will also consider systems of $m$ uniformly random polynomials $f_1, \cdots , f_m$ in $n$ variables over a general finite commutative ring $R$. In this case, only assuming the sample space for the $f_i$ is an $R$-module containing $1_R$, we show that the expected number of solutions to this system is $|R|^{n-m}$ (Theorem \ref{thm:ringmain1}).

We now introduce some notation and basic definitions. Throughout the paper, $\FF_q$ denotes the finite field of order $q=p^r$, and $R$ denotes a finite commutative ring with unity. We further denote the ring of polynomials in $n$ variables $x_1,\dots,x_n$ with coefficients in $R$ by $R[\overline{X}_n]$.

\begin{definition}
A point $\mathbf{p}\in R^n$ is called a \emph{zero} of a polynomial $f\in R[\overline{X}_n]$ if $f(\mathbf{p})=0$. If $f_1(\mathbf{p})=\cdots=f_m(\mathbf{p})=0$, then $\mathbf{p}$ is called a \emph{solution} of a system $f_1,\dots,f_m \in R[\overline{X}_n]$.
\end{definition}

In Sections \ref{sec:finitefieldsec} and \ref{sec:ringsec}, we will consider random polynomials sampled uniformly from $\FF_q$-subspaces and $R$-submodules of $\FF_q[\overline{X}_n]$ and $R[\overline{X}_n]$ respectively. When such a space is finite, “uniformly random” means with respect to the normalized counting measure; when it is infinite, such statements are interpreted in terms of asymptotic density, which will be discussed in more depth in Section \ref{sec:prelim}.

The following structural conditions on these spaces play a central role.

\begin{definition}\label{defn:spaces}
\begin{enumerate}[label=(\alph*), leftmargin=2em]
    \item An $\FF_q$-vector space $V \subseteq \FF_q[\overline{X}_n]$ is said to \emph{contain functions} if every function $\FF_q^n \to \FF_q$ is represented by a polynomial in $V$.
    \item A free $R$-module $K \subseteq R[\overline{X}_n]$ is said to \emph{extend $R$} if it contains the constant polynomials, equivalently if $1_R \in K$.
\end{enumerate}
\end{definition}
The first condition is equivalent to the surjectivity of the evaluation map
\[
V \longrightarrow \FF_q^{\FF_q^n}, \qquad
f \longmapsto \bigl(f(\mathbf{x})\bigr)_{\mathbf{x}\in\FF_q^n}.
\]
As we will see, the first condition is precisely what guarantees Theorem \ref{thm:main1}. On the other hand, the second condition ensures that the evaluation map $K \to R$ is surjective at each point, which is essential to the proof of Theorem \ref{thm:ringmain1}.

There are many natural examples of such spaces. For one, $\FF_q[\overline{X}_n]$ itself contains functions, as does the subspace of polynomials in $\FF_q[\overline{X}_n]$ with total degree at most $d \geq n(q-1)$, and the polynomials of degree at most $d \geq q-1$ in each variable. Moreover, replacing $\FF_q$ with a general ring $R$, all such spaces also extend $R$. More generally, for any $d \in \NN$, the $R$-submodule of polynomials with total degree at most $d$ extends $R$, as does the $R$-submodule of polynomials with degree at most $d$ in each variable.

The following two theorems are the main results of the paper.

\setcounter{maintheorem}{0}
\begin{maintheorem}\label{thm:main1}
Let $V \subseteq \FF_q[\overline{X}_n]$ be a vector space which contains functions. Then the number of distinct common zeros of a system of random polynomials $f_1,\dots,f_m \in V$ is binomial with parameters $q^n$ and $1/q^m$.
\end{maintheorem}

\setcounter{maintheorem}{1}
\begin{maintheorem}\label{thm:ringmain1}
Let $R$ be a finite commutative ring and let $K \subseteq R[\overline{X}_n]$ extend $R$. Then the expected number of distinct common zeros of a system of random polynomials $f_1,\dots,f_m \in K$ is
\[
|R|^{\,n-m}.
\]
\end{maintheorem}

When the sample spaces are infinite, these results are interpreted in terms of asymptotic density rather than probability. We discuss this distinction in Section~\ref{sec:prelim}.

\subsection{Related Work} There has been much interesting work on the distribution of the number of zeros of random polynomials. The question was first considered over the real numbers by Bloch and Polya \cite{BP32}, who studied polynomials with uniform coefficients $\{0, \pm 1 \}$. Following this, a breakthrough paper of Kac \cite{Kac43} showed that the expected number of real roots of a random polynomial of degree $n$ with independent standard Gaussian coefficients grows like $\frac{2}{\pi}\log n$ as $n \to \infty$. This work inspired a large body of subsequent research: one can find an excellent survey of these results in \cite{EK95}. For more recent work, see for instance \cite{DPSZ02, NV21}. 

Similar questions were considered over $p$-adic fields by \cite{Eva06, BCFG22}. Since $\ZZ_p$ arises as the inverse limit of the rings $\ZZ/p^k\ZZ$, it is natural to ask how probabilistic results over $\ZZ/p^k\ZZ$, in particular Theorem \ref{thm:ringmain1}, relate to corresponding questions over $\ZZ_p$.

Over finite fields, some prior work on the zeros of random polynomials include \cite{Odo92, Leo06, JMW25}. In \cite{Leo06}, it was shown that the distribution of the number of zeros of a random polynomial $f \in \FF_q[x]$ converges to a Poisson distribution with parameter $1$ as $q \to \infty$. Generalizing this result, in \cite{JMW25}, the distribution of the number of zeros of a random polynomial in $n$ variables over $\FF_q$ was computed. In the case $n = 1$, the result of Leont\'ev was recovered by taking the limit $q \to \infty$.

We greatly improve the results of \cite{Leo06, JMW25} in two ways. First, in Theorem \ref{thm:main1}, we compute the distribution of the number of solutions to a \emph{system} of random polynomials: moreover, we sample each polynomial from a maximally general class of sample spaces. One notable consequence is that the distribution of the number of solutions of a random system of $n$ polynomials over $\FF_q$ in $n$ variables is asymptotically Poisson with parameter $1$ as $q \to \infty$ (Corollary \ref{cor:poisson}).
Second, in Theorem \ref{thm:ringmain1}, we also compute the expectation for a system of polynomials over an arbitrary finite commutative ring.

\subsection{Organization of the paper} In Section \ref{sec:prelim}, we discuss the algebraic and probabilistic background for our discussion. In Section \ref{sec:finitefieldsec}, we prove Theorem \ref{thm:main1} using an elementary linear-algebraic argument. In Section \ref{sec:ringsec}, we prove Theorem \ref{thm:ringmain1} using an algebraic and combinatorial argument. Finally, in Section \ref{sec:conclusion}, we summarize the main results and give some suggestions for future work.

\section{Preliminaries}\label{sec:prelim}

In this section, we will outline the framework for the remainder of our discussion. 

We begin by recalling some algebraic facts. First, for a commutative ring $R$, an $R$-module $M$ is said to be \emph{free} if it has a basis, i.e. it is isomorphic to the direct sum $R^{(I)}$ for some index set $I$. In this case, $|I|$ is called the \emph{rank} of $M$, and is denoted by $\mathrm{rank} \ M$. If $|I|$ is finite, then $M$ is said to be \emph{finitely generated}, else it is said to be \emph{infinitely generated}.

Now, for a free $R$-module $M \subseteq R[\overline{X}_n]$ with rank $k \in \NN$, define 
\begin{equation}\label{eqn:X(M,n,m)}
X_R(M,n,m): M^m \to \ZZ_{\ge 0}
\end{equation} 
to be the function which, given a collection of polynomials $(f_1, \cdots , f_m) \in M^m$, outputs its number of distinct solutions. If $M$ is finitely generated, we can make $X_R(M,n,m)$ a random variable by equipping it with the normalized counting measure $\mu$ on $M^m$. Note that the choice of measure $\mu$ works because $M$ is finite by the finiteness of $R$, and it is canonical, as $\mu$ is the unique Haar probability measure on the additive group $M^m$ equipped with the discrete topology.

To extend this setup to the case of $M$ being infinitely generated, we will need to adopt a density function rather than a probability measure, as there is no uniform measure on an infinite set. For $t \in \ZZ_{\ge 0}$, let $M_t \subseteq M$ denote the $R$-submodule consisting of polynomials of total degree at most $t$. Then, note that each $M_t$ is finite and free over $R$, and the sequence $\{M_t\}_{t\ge 0}$ exhausts $M$. Then, given a subset $A \subseteq M^m$, we define its \emph{density} by
\begin{equation}\label{eq:mdensity}
d_m(A)
\;=\;
\lim_{t\to\infty}
\frac{|A \cap M_t^m|}{|M_t^m|},
\end{equation}
provided the limit exists.

Note that if $M$ extends $R$, so will the $M_t$, and if $M$ contains functions, then so too will the $M_t$ for $t$ sufficiently large. This allows us to translate results proved in the finite-dimensional setting by only considering polynomials in $M^m$ with total degree at most $t$, then taking the limit $t \to \infty$.

We also note that $d_m$ is only finitely additive, so it is not a measure in the usual sense. Therefore, $(X_R(M,n,m), d_m)$ is not a random variable if $M$ is infinite: instead, it is simply a function which measures the density of systems of polynomials in $M^m$ with $k$ distinct solutions for all $k$ with respect to $d_m$. Nevertheless, by abuse of notation, we will treat it as such in the proofs of Theorems \ref{thm:main1} and \ref{thm:ringmain1} in order to unify the treatment of the finitely-generated and infinitely-generated cases.


\section{Proof of Theorem \ref{thm:main1}}\label{sec:finitefieldsec}
In this section, we will prove Theorem \ref{thm:main1}, hence computing the distribution of the number of common zeros of a system of random polynomials $f_1, \cdots , f_m$ sampled from a vector space $V \subseteq \FF_q[\overline{X}_n]$ which contains functions.

Recall that a vector space $V \subseteq \FF_q[\overline{X}_n]$ is said to contain functions if every function $\FF_q^n \to \FF_q$ is represented by a polynomial in $V$. After proving Lemma \ref{lem: lem1}, we will give a necessary-and-sufficient criterion (Proposition \ref{prop:nec}) for this condition to hold.

\begin{lemma}\label{lem: lem1}
Every function $\FF_q^n \to \FF_q$ is uniquely represented by a 
polynomial in the space $\FF_q[\overline{X}_n]_{\overline{q-1}} \subset \FF_q[\overline{X}_n]$ of polynomials with degree at most $q-1$ with respect to each variable $x_i$.
\end{lemma}
\begin{proof}
Let $S$ denote the set of all functions $\FF_q^n \to \FF_q$. We will show that the function $\varphi: \FF_q[\overline{X}_n]_{\overline{q-1}} \to S$ is injective, where $\varphi(f): \FF_q^n \to \FF_q$ is the function induced by a polynomial $f$. Suppose that $f,g \in \FF_q[\overline{X}_n]_{\overline{q-1}}$ induce the same function. Then, $f-g$ induces the zero function, that is, $f-g$ is a member of the ideal
\begin{equation}\label{eqn:eqnI}
 I := (x_1^q - x_1, \cdots, x_n^q - x_n). 
\end{equation}
But if $f-g \in I$, then the degree of $f-g$ is at least $q$ with respect to some variable $x_i$, or $f-g = 0$. Since every polynomial in the ring $\FF_q[\overline{X}_n]_{\overline{q-1}}$ has degree at most $q-1$ with respect to each variable, we have $f-g = 0$, or $f = g$, implying that $\varphi$ is injective.

Now, $\FF_q[\overline{X}_n]_{\overline{q-1}}$ has dimension $(q-1+1)^n = q^n$, therefore $|\FF_q[\overline{X}_n]_{\overline{q-1}}| = q^{q^n}$. On the other hand, $|S| = |\FF_q|^{|\FF_q^n|} = q^{q^n}$, hence $|\FF_q[\overline{X}_n]_{\overline{q-1}}| = |S|$. Since $\varphi$ is an injective function between finite sets of equal cardinality, it must be a bijection.
\end{proof}

In some sense, $\FF_q[\overline{X}_n]_{\overline{q-1}}$ is the "smallest" vector space which contains functions. 

\begin{proposition}\label{prop:nec}
A vector space $V \subseteq \FF_q[\overline{X}_n]$ contains functions if and only if for every $f \in \FF_q[\overline{X}_n]_{\overline{q-1}}$, there exists a polynomial $g \in V$ such that $f- g \in I$, where $I$ is defined as in (\ref{eqn:eqnI}).
\end{proposition}

\begin{proof}
First, if $V$ contains functions, then for any $f \in \FF_q[\overline{X}_n]_{\overline{q-1}}$, there exists a polynomial $g \in V$ which induces the same function $\FF_q^n \to \FF_q$ as $f$. Therefore, $f-g \in I$ as desired. On the other hand, if for every $f \in \FF_q[\overline{X}_n]_{\overline{q-1}}$, there exists a polynomial $g \in V$ such that $f- g \in I$, then every function $f:\FF_q^n \to \FF_q$ is represented by $g \in V$.
\end{proof}


\begin{lemma}\label{lem: lem2}
Let $\bp_1, \cdots \bp_{r}$ be distinct points in $\FF_q^n$, and let $V \subseteq \FF_q[\overline{X}_n]$ be a vector space. Then, the linear map
\[
\phi_r: V \to \FF_q^{r} \qquad \phi_r(f) = (f(\bp_{1}), \cdots , f(\bp_{r}))
\]
is surjective for all $1 \le r \le q^n$ exactly when $V$ contains functions.
\end{lemma}

\begin{proof}
First, assume $V$ contains functions, and $r = q^n$. Then, by definition, $V$ contains a polynomial representing every function $\FF_q^n \to \FF_q$. In particular, for any $(a_1, \cdots , a_{q^n}) \in \FF_q^{q^n}$, there exists a polynomial $f$ such that
\[
f(\bp_i) = a_i
\]
for all $1 \le i \le q^n$. It follows immediately that $\phi_{q^n}$ is surjective. Now, suppose that $r < q^n$. Then, for any point $(a_1, \cdots , a_r) \in \FF_q^r$, as shown above, there exists a polynomial $f \in V$ such that
\begin{equation}\label{eqn: surj}
f(\bp_i) = a_i
\end{equation}
for all $1 \le i \le r$, and 
$
f(\bp) = \omega
$
for all $\bp \in \FF_q^n \setminus \{ \bp_1, \cdots \bp_{r} \}$ and an arbitrary $\omega \in \FF_q$. In particular, (\ref{eqn: surj}) implies that $\phi_r$ is surjective. On the other hand, assuming $V$ does not contain functions, the map $\phi_{q^n}$ is clearly not surjective, therefore $\phi_r$ is not surjective for all $1 \le r \le q^n$. 
\end{proof}

\begin{lemma}\label{lem: lem3}
Let $f \in V$ be a random polynomial sampled from  a finite-dimensional vector space $V \subset \FF_q[\overline{X}_n]$ which contains functions. Then, for any $\bp \in \FF_q^n$,
\[
P \left( f(\bp) = 0 \right) = \frac{1}{q}.
\]
\end{lemma}

\begin{proof}
For all $\bp \in \FF_q^n$, define the evaluation map
\[
\mathrm{ev}_{\bp}: V \to \FF_q, \qquad \mathrm{ev}_{\bp}(f) = f(\bp).
\]

Since the map $\mathrm{ev}_{\bp}$ is clearly surjective, the space $\ker \ \mathrm{ev}_{\bp}$ of polynomials which vanish at $\bp$ has dimension $\dim V - 1$ by the rank-nullity theorem. Therefore,

\[
P \left( f(\bp) = 0 \right) = \frac{|\ker \ \mathrm{ev}_{\bp}|}{|V|} = \frac{q^{\dim V - 1}}{q^{\dim V}} = \frac{1}{q}.
\]
\end{proof}

\begin{lemma}\label{lem: niceref}
Let $V \subset \FF_q[\overline{X}_n]$ be a finite-dimensional vector space. Then, for any positive integer $r \le q^n$, there are $q^{\dim V - r}$ polynomials in $V$ which vanish at any $r$ points $\bp_1 \cdots \bp_r \in \FF_q^n$ if and only if $V$ contains functions.
\end{lemma}

\begin{proof}
Note that the space of polynomials vanishing on $\bp_1 \cdots \bp_r \in \FF_q^n$ is precisely the kernel of the map
\[
V \to \FF_q^{r} \qquad \phi_r(f) = (f(\bp_{1}), \cdots , f(\bp_{r})).
\]
By the rank-nullity theorem,
\[
\dim \ker \phi_r= \dim V - \dim \mathrm{im} \ \phi_r.
\]
Now, by Lemma \ref{lem: lem2}, we have
\begin{equation}\label{eqn:kerim}
\dim \ker \phi_r= \dim V - r
\end{equation}
exactly when $V$ contains functions. Finally, $\ker \phi_r$ is a vector space over $\FF_q$, so (\ref{eqn:kerim}) implies 
\[
|\ker \phi_r| = q^{\dim V -r}
\]
as desired.
\end{proof}

The following corollary states that if $f$ is a random polynomial from a finite-dimensional vector space $V$ which contains functions, then the events of $f$ vanishing at any distinct points $\bp_1, \cdots , \bp_r$ are independent.

\begin{corollary}\label{cor: nice}
Let $V \subset \FF_q[\overline{X}_n]$ be a finite-dimensional vector space. For a random polynomial $f \in V$, define $E_j$ as the event of $f$ vanishing on a fixed $\bp_j \in \FF_q^n$. Then, for any $I \subseteq \{1, \cdots, q^n \}$, the events $\{E_i \}_{i \in I}$ are independent exactly when $V$ contains functions.
\end{corollary}

\begin{proof}
By Lemma \ref{lem: niceref}, for any $I := \{a_1, \cdots, a_r \} \subseteq \{1, \cdots, q^n \}$, we have 
\[
P(E_{a_1}, \cdots , E_{a_r}) = \frac{q^{\dim V-r}}{q^{\dim V}} = q^{-r}, 
\]
and $P(E_{a_i}) = q^{-1}$ for any $a_i \in I$. Combining the two, \[ P(E_{a_1}, \cdots , E_{a_r}) = \prod_{i=1}^{r} P(E_{a_i}), \] implying that the events $\{E_i \}_{i \in I}$ are independent as desired.
\end{proof}

We are ready to prove the main result.

\noindent
\textbf{Proof of Theorem \ref{thm:main1}.} 
We first assume that $V$ is finite-dimensional. Let $\FF_q^n = \{\bp_1, \cdots \bp_{q^n} \}$, and for all $1 \le j \le  q^n$, define the indicator random variable

\[
    X_{j}  = \begin{cases}
        1 & f_1(p_j) = \cdots = f_m(p_j) = 0,\\
        0, & \text{otherwise}.
    \end{cases}
\]
Note that the $X_{j}$ are Bernoulli trials, where $X_{j} = 1$ if the $f_1, \cdots , f_m$ simultaneously vanish at $\bp_j$, and $X_{j} = 0$ if they do not. We claim that the random variables $X_1, \cdots , X_{q^n}$ are iid with parameter $1/q^{m}$. 

First, we will show that, for any $1 \le j \le q^n$, $\ X_j$ has parameter $1/q^m$. By Lemma \ref{lem: lem3}, each polynomial $f_i$ vanishes at $\bp_j$ with probability $1/q$. Since the $f_i$ are chosen independently, we have
\[
P(f_1(\bp_j) = 0, \cdots ,  f_m(\bp_j) = 0) = \prod_{i=1}^{m}P(f_i(\bp_j)) = \frac{1}{q^m}
\]
as desired. On the other hand, the $X_i$ are also independent, since, by the same reasoning as before, the events of two distinct random polynomials $f_{a_1}, f_{a_2}$ vanishing at any point $\bp \in \FF_q^n$ are independent of one another, and by Corollary \ref{cor: nice}, the events of a fixed random polynomial $f_i$ vanishing at any set of points $\bp_1, \cdots , \bp_r \in \FF_q^n$ are independent. The claim is proven, and because
\[
\sum_{j=1}^{q^n}X_i
\]
is exactly the number of common zeros of $f_1, \cdots , f_m$, the result is also shown.

Now, suppose that $V$ is infinite-dimensional. For all $r \in \ZZ_{\ge 0}$, define the set
\[
P_{m,r} = \{(f_1, \cdots , f_m) \in V^m \mid f_1, \cdots , f_m \text{ have $r$ common distinct zeros} \}.
\]
Note that
\[
|P_{m,r} \ \cap \ V_t^m| = |\{(f_1, \cdots , f_m) \in V_t^m \mid f_1, \cdots , f_m \text{ have $r$ common zeros} \}|,
\]
therefore
\[
d_m(P_{m,r}) = \lim_{t \to \infty}\frac{|P_{m,r} \ \cap \ V_t^m|}{|V_t^m|} = \lim_{t \to \infty} P(X_{\FF_q}(V_t, n, m) = r).
\]

By assumption, there exists some $\ell \in \NN$ such that for all $t \ge \ell$, $V_t$ contains functions, thus
\[
\lim_{t \to \infty} P(X_{\FF_q}(V_t, n, m) = r) = \binom{q^n}{r} \left( 1/q^m \right)^r \left( 1 - 1/q^m \right)^{q^n - r}
\]
by the finite-dimensional case. \qed


\begin{remark}\label{rmk: ringrmk}
One may wonder if there exists a finite commutative ring $R$ which is not a field such that the distribution of $X_R(M,n,m)$ is binomial with parameters $|R|^n$ and $1/|R|^m$ for some $M \subseteq R[\overline{X}_n]$. The answer to this question is negative for the following reason. Unless $R$ is a finite field, there will always be a function $R^n \to R$ which cannot be represented by a polynomial in $R[\overline{X}_n]$. For example, define the function $\eta: R \to R$ as
\[
    \eta(r)  = \begin{cases}
        0_R & r = 0_R,\\
        1_R, & r \ne 0_R.
    \end{cases}
\]
If $\eta$ were represented as a polynomial $\eta(x) = a_1x + \cdots a_nx^n$, we would have
\[
\eta(r) = r(a_1 + \cdots + a_nr^{n-1}) = 1_R
\]
for any nonzero $r$. That is, every nonzero element in $R$ would be invertible, making $R$ a finite field \cite{Fri99}. Therefore, there are no $R$-modules $M \subseteq R[\overline{X}_n]$ which contain functions unless $R$ is a field.
\end{remark}

\begin{corollary}\label{cor: bincor}
If $V \subseteq \FF_q[\overline{X}_n]$ contains functions, then,
\[
X_{\FF_q}(V, n, m) \sim \mathrm{Bin} \left( q^n, \frac{1}{q^m} \right).
\]
\end{corollary}

\begin{proof}
This is immediate from Theorem \ref{thm:main1}.
\end{proof}

\begin{corollary}\label{cor:poisson}
As $q \to \infty$, the number of common zeros of a system of $n$ (not-necessarily distinct) random polynomials in $\FF_q[\overline{X}_n]$ follows a Poisson distribution with parameter $1$.
\end{corollary}

\begin{proof}
By Corollary \ref{cor: bincor}, as $q \to \infty$, the asymptotic density of the systems of polynomials in $\FF_q[\overline{X}_n]^m$ with $k$ solutions is given by
\[
\lim_{q \to \infty} P\left( X_{\FF_{q}}(\FF_{q}[\overline{X_n}], n, m) = k  \right) = \lim_{q \to \infty} \frac{e^{-q^{n-m}}q^{k(n-m)}}{k!} = \frac{e^{-1}}{k!}
\]
as desired.
\end{proof}

\section{Proof of Theorem \ref{thm:ringmain1}}\label{sec:ringsec}

In this section, we will prove Theorem \ref{thm:ringmain1}, computing the average number of zeros of a random system of polynomials in an $R$-module $M \subseteq R[X_n]$ that extends $R$. 

\begin{lemma}\label{lem: Np}
Suppose $K \subseteq  R[\overline{X}_n]$ is a finitely generated $R$-module that extends $R$ of rank $k$. Then, for each $\bp \in R^n$, the space
\[
N_\bp := \left\{ f \in K \mid f(\bp) = 0 \right\}
\]
of polynomials in $K$ which vanish at $\bp$ has cardinality $|R|^{k - 1}$. 
\end{lemma}

\begin{proof}
First, since $R$ is contained in $K$, $K$ contains the multiplicative identity $1_R$ of $R$. In particular, the linearly independent set $\{1_R\}$ can be extended to a basis
\[
\mathcal{B}_K = \{ 1_R, r_1, \cdots , r_{k-1} \} \subset K
\]
of $K$. Now, choose $c_1, \cdots , c_{k-1} \in R$ such that, for all $1 \le i \le k-1, \ r_i - c_i$ vanishes at $\mathbf{0}$. We claim that
\[
\mathcal{B}_K' = \{ 1_R, r_1-c_1, \cdots , r_{k-1}-c_{k-1} \} \subset K
\]
is a basis of $K$. First, to show that the elements in $\mathcal{B}_K'$ are linearly independent, suppose 
\[
a_0 \cdot 1_R + \sum_{i=1}^{k-1}a_i(r_i - c_i) = 0
\]
for some $a_0, \cdots , a_{k-1} \in R$. Then, rearranging,
\begin{equation}\label{eqn:indep}
\left( a_0 -\sum_{i=1}^{k-1}a_ic_i \right) \cdot 1_R + \sum_{i=1}^{k-1}a_ir_i = 0.
\end{equation}

Since $1_R, r_1, \cdots, r_{k-1}$ are linearly independent, (\ref{eqn:indep}) implies
\[
a_0 -\sum_{i=1}^{k-1}a_ic_i = 0, \qquad a_1= \cdots = a_{k-1} = 0.
\]
But $\sum_{i=1}^{k-1}a_ic_i = 0$, so we also have $a_0 = 0$. Therefore, $\mathcal{B}_K'$ is a linearly independent set with cardinality $\mathrm{rank} \ K = k$. In particular, $\mathcal{B}_K'$ is a basis of $K$, proving the claim.

Then, note that the $R$-module $N_{\mathbf{0}}$ of polynomials which vanish at $\mathbf{0} := (0, \cdots , 0)$ is generated by $\mathcal{B}_K' \setminus \{1_R \}$, so $\mathrm{rank} \ N_{\mathbf{0}} = k-1$. Since $N_{\mathbf{0}}$ is an $R$-module, this implies
\[
|N_{\mathbf{0}}| = |R|^{k-1}.
\]
Then, for all $\bp \in R^n$, define 
\[
\phi_\bp: N_{\mathbf{0}} \to N_\bp, \qquad \phi_\bp(f) := f - f(\bp).
\]
Since $\phi_\bp$ has an inverse map
$\phi^{-1}_\bp(g) := g - g(\mathbf{0})$, it is bijective. Therefore, $|N_\bp| = |R|^{k - 1}$ for all $\bp \in R^n$.
\end{proof}

\begin{lemma}\label{lem: ringvanishingprob}
Let $K \subseteq R[\overline{X}_n]$ be a finitely generated $R$-module with rank $k$ that extends $R$. For a random polynomial $f \in K$ and any point $\bp \in R^n$, we have
\[
P(f(\bp) = 0) = \frac{1}{|R|}.
\]
\end{lemma}

\begin{proof}
Defining $N_\bp := \left\{ f \in K \mid f(\bp) = 0 \right\}$ as above, $|N_{\mathbf{\bp}}| = |R|^{k-1}$  by Lemma \ref{lem: Np}, we have
\[
P(f(\bp) = 0) = \frac{|N_\bp|}{|K|} = \frac{|R|^{k - 1}}{|R|^{k}} = \frac{1}{|R|}
\] 
as desired.
\end{proof}

\begin{lemma}\label{lem: ringvanishingprob2}
Let $M \subseteq R[\overline{X}_n]$ be an infinitely generated $R$-module that extends $R$, and define $N_{\mathbf{p}} := \{f \in M \mid f( \mathbf{p}) = 0 \}$. Then, for all $\mathbf{p} \in R$, the density of $N_{\mathbf{p}}$ in $M$ is $1/|R|$.
\end{lemma}

\begin{proof}
We have
\[
d_1(N_{\mathbf{p}}) = \lim_{t \to \infty}\frac{|N_{\mathbf{p}} \ \cap \ M_t|}{|M_t|} = \lim_{t \to \infty}\frac{\{f \in M_t \mid f( \mathbf{p}) = 0 \}|}{|R|^t} = \lim_{t \to \infty}\frac{R^{t-1}}{|R|^t} = \frac{1}{|R|}
\]
by Lemma \ref{lem: ringvanishingprob}.
\end{proof}

\noindent
\textbf{Proof of Theorem \ref{thm:ringmain1}.} 
For all $\bp \in R^n$, define the random variable
\[
    V_{\bp}  := \begin{cases}
        1 & f_1(\bp) = \cdots = f_m(\bp) = 0,\\
        0, & \text{ otherwise}.
    \end{cases}
\]
Since the polynomials $f_1, \cdots , f_m$ are selected independently, the events $f_i(\bp) = 0$ and $f_j(\bp) = 0$ are independent for all $i \ne j$, and
\[
P \left(V_{\bp} = 1 \right) = \prod_{i=1}^{m}P(f_i(\bp )=0).
\]

If $K$ is finitely generated, then by Lemma \ref{lem: ringvanishingprob},
\[
P(f_i(\bp )=0) = \frac{1}{|R|},
\]
and if $K$ is infinitely generated, then by Lemma \ref{lem: ringvanishingprob2},
\[
P(f_i(\bp )=0) = d_1(N_{\mathbf{p}}) = \frac{1}{|R|}
\]
for all $1 \le i \le m$. Therefore, $P \left(V_{\bp} = 1 \right) = 1/|R|^m$ in both cases.

Now, the number of common zeros of the $f_i$ is given by $\sum_{\bp \in R^n}V_{\bp}$, thus by the linearity of expectation, 
\begin{equation}\label{eqn:ringexp}
\mathbb{E} \left( \sum_{\bp \in R^n}V_{\bp} \right) = \sum_{\bp \in R^n}\mathbb{E}\left( V_{\bp} \right).   
\end{equation}

Since $\mathbb{E}\left( V_{\bp} \right) = P \left(V_{\bp} = 1 \right) = 1/|R|^m$, (\ref{eqn:ringexp}) simplifies to
\[
 \sum_{\bp \in R^n} \frac{1}{|R|^m} = |R|^{n-m}
\]
as desired. \qed

\begin{remark}
For all $0 \le r \le |R|^n$, let $A_r \subseteq K^m$ denote the set of $m$-tuples of polynomials with exactly $r$ common zeros. If $K$ is finitely generated, Theorem \ref{thm:ringmain1} states
\[
\frac{1}{|R|^n+1}\sum_{r=0}^{|R|^n}\mu(A_r) = |R|^{n-m},
\]
and if $K$ is infinitely generated, Theorem \ref{thm:ringmain1} states
\[
\frac{1}{|R|^n+1}\sum_{r=0}^{|R|^n}d_m(A_r) = |R|^{n-m}.
\]
\end{remark}

\section{Future Work}\label{sec:conclusion}
In the future, it will be interesting to explore applications of the above results in cryptography. For instance, the security of many cryptography algorithms depend on the NP-completeness of the problem of finding solutions to a random system of polynomials over a finite field \cite{LPTV24}. While the above results do not provide a deterministic method of finding the solutions to a system of polynomials, they provide a good heuristic for the number of solutions we can expect. For example, according to Theorem \ref{thm:main1}, three random polynomials $f_1, f_2 , f_3 \in \FF_{19}[\overline{X}_2]$ has at most one solution with probability $\approx 0.9987$. Therefore, if we obtain a single solution, it is very likely to be the only one. 

Multivariable polynomials over finite fields and rings are also of pure mathematical interest, being linked to the factorization of random polynomials over $\QQ$, polynomials over the $p$-adic integers $\ZZ_p$ via the inverse limit $\varprojlim \ZZ/p^k \ZZ = \ZZ_p$, and the geometry of complex varieties. See \cite{BV19, Shm21, Del74}. In the future, it will be interesting to explore these connections further.

It might be possible to obtain more geometrically natural specializations of Theorems \ref{thm:main1} and \ref{thm:ringmain1}. For example, what is the distribution of the number of $\FF_q$-rational points on a random \emph{complete intersection} $V \subset \AA^{n}(\FF_q)$? How many rational points can we expect on a projective variety $V \subset \PP^n(\FF_q)$ defined by $m$ random homogenous polynomials $f_1, \cdots , f_m \in \FF_q[\overline{X_n}]$? In \cite{JMW25}, the latter expectation was found to be approximately $q^{n-1}$ in the case of $m = 1$ polynomials, essentially relying on the argument given in the proof of Lemma \ref{lem: lem3}. Can we more generally compute the expectation and probability distribution in the case of projective varieties for all $m \in \NN$?

\acknowledgement
The author would like to thank Han-Bom Moon for many helpful conversations, and his continued support in the writing of this paper. He is also grateful to Fordham University, at which this work was completed.

\end{document}